\documentclass[12pt]{amsart}

\usepackage{algorithm}
\usepackage{algorithmic}

\usepackage{graphicx}
\usepackage{rotating}

\usepackage{color}
\usepackage[dvips]{xcolor}

\definecolor{Blue}{rgb}{0.3,0.3,0.9}
\definecolor{orange}{rgb}{1,0.5,0}

\usepackage{verbatim}

\newtheorem{theorem}{Theorem}[section]
\newtheorem{lemma}[theorem]{Lemma}

\newtheorem{remark}[theorem]{Remark}

\usepackage{epsfig}
\newcommand{\bi}{\partial} 

\def\hii{{h}}
\def\hij{{h}}

\def\Xi{{X}_h}

\begin{document}

\title[Generalized Multiscale FEM. SIPG coupling]
{Generalized Multiscale Finite Element Method. Symmetric Interior Penalty Coupling}

\author{Y. Efendiev} 
\address{Department of Mathematics, Texas A\&M University, College Station, TX 77843, USA}
\email{efendiev@math.tamu.edu}

\author{J. Galvis}
\address{Departamento Matem\'aticas, Universidad Nacional de Colombia,
Carrera 45 No 26-85 - Edificio Uriel Gutierr\'ez, 
Bogot\'a D.C. - Colombia}
\email{jcgalvisa@unal.edu.co}

\author{R. Lazarov} 
\address{Department of Mathematics, Texas A\&M University, College Station, TX 77843, USA}
\email{lazarov@math.tamu.edu}

\author{M. Moon} 
\address{Department of Mathematics, Texas A\&M University, College Station, TX 77843, USA}
\email{mmoon@math.tamu.edu}

\author{M. Sarkis}
\address{Mathematical Sciences Department,
Worcester Polytechnic Institute, 
100 Institute Road, 
Worcester, MA 01609-2280. AND 
Instituto Nacional de Matem\'atica Pura e Aplicada, Estrada Dona Castorina 110, CEP 22460-320, Rio de Janeiro,
Brazil}
\email{msarkis@wpi.edu, msarkis@impa.br }






\maketitle


\begin{abstract}
 Motivated by applications to numerical simulation of flows in highly heterogeneous porous media,  we develop multiscale finite element methods for second order elliptic equations.
We discuss a multiscale model reduction technique in the framework of the discontinuous  Galerkin finite element method. 
We propose  three different finite element spaces on the coarse mesh. 
 The first space  is based on a local eigenvalue problem that uses a weighted $L_2-$norm for computing the "mass" matrix.  
The second space is generated by amending the eigenvalue problem of the first case with a term related to the penalty. 
The third choice is based on  generation of a large space of snapshots and subsequent selection of a subspace of a reduced dimension.
  The approximation with these spaces is based on the discontinuous  Galerkin finite element method framework. 
We investigate the stability  and derive error estimates for the methods  and further experimentally study their performance on a representative number of numerical examples.
\end{abstract}

\section{Introduction}\label{sec:intro}

In this paper we present a study of numerical methods for the
simulation of flows in highly heterogeneous porous media. 
The media properties are assumed to
contain multiple scales and high contrast. In this case, 
solving the systems arising in the approximation of
 the flow equation on a fine-grid  that resolves all scales  
by the finite element, finite volume, or mixed FEM
could be prohibitively expensive, unless special care is taken for solving the resulting system.
A number of techniques have been proposed to efficiently solve the 
these fine-grid systems.
Among these are  
multigrid  methods (e.g., \cite{Arb_MG,Eb_Wittum_05}),
multilevel methods  (e.g., \cite{Vassilevski_book_08, Vass_Upscal_11}),
 and domain decomposition techniques
(e.g., \cite{ ge09_2, Efendiev_GLW_ESAIM_12, ge09_1, Graham1, hw97,  tw}).

More recently, a new large class of accurate reduced-order methods
has been introduced and used in various applications. These include 
Galerkin multiscale finite elements 
(e.g., \cite{Arbogast_two_scale_04, Chu_Hou_MathComp_10,ee03,egw10,eh09,ehg04}),
mixed multiscale finite element methods 
(e.g., \cite{aarnes04, ae07, Arbogast_Boyd_06,Iliev_MMS_11}), 
the multiscale finite volume method (see, e.g., 
\cite{jennylt03}), 
mortar multiscale methods (see e.g., \cite{Arbogast_PWY_07, Wheeler_mortar_MS_12}),  
and variational multiscale methods
(see e.g., \cite{hughes98}).
Our main goal is to extend these concepts and develop 
a systematic methodology for solving complex multiscale problems with
high-contrast and no-scale separation by using discontinuous basis functions.

In this paper, we study the 
 multiscale model reduction techniques within discontinuous Galerkin framework.
As the  problem is expected to be solved for many input parameters such
as source terms, boundary conditions, and spatial heterogeneities, 
we divide the computation into two stages 
(following known formalism \cite{ boyaval, maday}): 
offline and online, where  our goal in the offline stage is
to construct a reduced dimensional multiscale space to be used for rapid 
computations in the online stage.
In the offline stage \cite{egt11}, we generate a snapshot space 
and propose a local spectral problem that allows selecting dominant modes in the
space of snapshots. 
In the online stage use the basis functions 
computed offline to solve the problem for current realization of the parameters (a further spectral 
selection may be done in the online step in each coarse block).  As a result,
the basis functions generated by coarse block computations are discontinuous
along the coarse-grid inter-element faces/edges.
Previously, e.g. \cite{egt11}, in order to generate 
conforming basis functions, partition of unity functions have been used. However, this procedure modifies original spectral
basis functions and is found to be difficult to apply for more complex flow problems. 
In this paper, we propose and explore the use of local model reduction
techniques within the framework of the 
discontinuous Galerkin finite element methods.

We introduce a Symmetric Interior Penalty Discontinuous Galerkin (SIPG) method that uses 
spectral basis functions that are constructed in special way in order 
to reduce the degrees of freedom of the local (coarse-grid) approximation spaces. 
Also we discuss the use of penalty parameter in the SIPG method  and derive a stability
result for a penalty that scales as the inverse of the fine-scale
mesh. We show that the stability constant is 
independent of the contrast. The latter is important as the problems
under consideration have high contrast.

We also derive error estimates and  discuss the 
convergence issues of the method. Additionally, 
the efficacy of the proposed methods is demonstrated on 
a set of numerical experiments with
flows in high-contrast media where the permeability
field has subregions of high conductivity, which form channels and islands. In both 
cases we observe that as  the dimension of the coarse-grid space increases, 
the error decreases and the decrease is proportional to the eigenvalue
that the corresponding eigenvector is not included in the coarse space.
In particular, we present results when the snapshot space consists of local solutions.

The paper is organized in the following way. In Section \ref{sec:prob} we present
our model problem in a weak form and introduce the approximation method that
involves two grids, fine (that resolves all scales of the heterogeneity) and
coarse (where the solution will be sought). On each cell of the coarse mesh we
introduce a lower dimensional space of functions that are defined on the fine mesh.
We also show that the method is stable in a special DG norm.
In Section \ref{sec:spaces} we present three different choices of local spaces.
The first two are based on few eigenfunctions of special  spectral problems
in the style of  \cite{ge09_2,ge09_1}. 
The third choice is based on the concept of snapshots, 
e.g. \cite{ boyaval, maday}.
In Section \ref{sec:numerics}  we present some numerical experiments  and report 
 the error of the Discontinuous Galerkin method with the
constructed coarse-grid spaces and in Section \ref{sec:discuss}, we
discuss the numerical  results.
The theoretical results are derived under the assumption that  the penalty 
stabilization depends on the fine-mesh size.
Based on the numerical experiments we can conclude that the interior penalty Galerkin method
gives reasonable practical results in using coarse-grid spaces generated by special problems, 
solved locally on each coarse-grid block, that take into account the  highly heterogeneous
behavior of the coefficient of the differential equation (in our case, the permeability).

\section{Continuous and discrete problems}\label{sec:prob}

We consider the following problem: Find $u^* \in H^1_0(\Omega)$ such that
\begin{equation} \label{eq:diff}
a(u^*, v) = f(v) \quad ~~\mbox{ for all } v \in H^1_0(\Omega)
\end{equation}
where
\[
a(u, v) :=  \int_{\Omega} \!\!\! \kappa(x) \nabla u \cdot \nabla v dx
~~\mbox{and} ~~ f(v) := \int_\Omega \!\!f v dx.
\]
Here $\Omega$ is a bounded domain in $R^d$, $d=2,3$ with polygonal boundary. We assume that $f \in L_2(\Omega)$ and
the coefficient $\kappa(x)$ represents the permeability of a highly heterogeneous porous media with high
contrast, that is high ratio between the maximum and minimum values, see Figure \ref{fig:coef-perm2_2}.
%
Our main goal in this paper is to develop an approximation method for 
(\ref{eq:diff})  on a coarse grid using certain ``low energy"
local eigenfunctions.

We consider the two dimensional case. The method and results presented here 
extend for three dimensional case.
We split the domain $\Omega$ into disjoint polygonal subregions $\{\Omega_i\}_{i=1}^N$
of diameter $O(H_i)$ so that  $\overline{\Omega} = \cup^N_{i=1} \overline{\Omega}_i$.
We assume that  the substructures  $\{\Omega_i\}_{i=1}^N$ form
a geometrically nonconforming partition of $\Omega$.
In this case,  for  $ i \neq j$,  the intersection  $\partial \Omega_i \cap \partial \Omega_j$
is either empty, a vertex of $\Omega_i$ and/or $\Omega_j$, or a common
edge  of $\partial \Omega_i$ and $\partial \Omega_j$.
We recall that in the case of geometrically conforming decomposition, the
intersection $\partial \Omega_i \cap \partial \Omega_j$
is either empty or a common vertex of $\Omega_i$ and $\Omega_j$,
or a common edge of $\Omega_i$ and $\Omega_j$. Similar construction is 
assumed in 3-D with  $\Omega_i$ being polyhedra.


Further, in each $\Omega_i$ we
introduce a shape regular triangulation $\mathcal{T}_{\hii}(\Omega_i)$
with triangular elements and maximum mesh-size $h_i$. The resulting triangulation
of $\Omega$ is in
general nonmatching across $\partial \Omega_i$. Let
$\Xi(\Omega_i)$ be
the regular finite element  space of piecewise linear and continuous functions in
$\mathcal{T}_{\hii}(\Omega_i)$.
We do not assume that functions in $\Xi(\Omega_i)$  vanish
 on $\partial \Omega_i
 \cap \partial \Omega $. We define
\[
X_h(\Omega) = \Xi(\Omega_1) \times \cdots \times \Xi(\Omega_N)
\]
and represent functions $v$ of  $X_h(\Omega)$ as
$v = \{v_i\}^N_{i=1}$ with $v_i\in \Xi(\Omega_i)$. For simplicity, we also assume that
the permeability $\kappa(x)$ is constant over each fine-grid element.

Due to the fact that $\mathcal{T}_{\hii}(\Omega_i)$ and
$\mathcal{T}_{\hij}(\Omega_j)$ are independent from each
other on a common edge $E=\partial \Omega_i \cap \partial \Omega_j$
they may introduce two different partitions of $E$ which are merged to obtain 
a set of faces $E_{ij} \subset E$. Since the functions in $X_h(\Omega)$ are 
discontinuous along the interfaces, it is necessary to distinguish between
$E\subset \overline{\Omega}_i$ and
$E\subset\overline{\Omega}_j$.   From now on 
the $\Omega_i$-side of $E$ will be denoted by $E_{ij}$ while 
on the $\Omega_j$-side of $E$ will be denoted by $E_{ji}$.
Geometrically, $E_{ij}$ and $E_{ji}$ are the same
object.

We  use the following  harmonic averages along the edges $E_{ij}$. For
$i,j\in\{1,\dots,N\}$ define
\begin{equation}\label{eq:def:averages}
\kappa_{ij} = \frac{2 \kappa_i \kappa_j}{\kappa_i + \kappa_j}
\quad
\mbox{ and }\quad  h_{ij} = \frac{2 h_i h_j}{h_i + h_j}.
\end{equation}
Note, that the functions $\kappa_{ij} $ and $h_{ij}$ are piecewise constants over the edge $E_{ij}$ on a mesh that
is obtained by merging the partitions $\mathcal{T}_{\hii}(\Omega_i)$ and
$\mathcal{T}_{\hij}(\Omega_j)$ along their common edge $E_{ij}$.

The discrete problem obtained by the DG method, see \cite{ABCM_unified_2002,Dryja_DG_03} 
 is: Find $u^*_h= \{u^*_{h,i}\}^N_{i=1} \in X_h(\Omega)$, 
$ u^*_{h,i} \in X_h(\Omega_i)$, such that
\begin{equation}\label{eq:disc}
{a}^{DG}_h(u^*_h, v_h) = f(v_h) \quad \mbox{ for all }
\quad v_h=
\{v_{h,i}\}^N_{i=1} \in X_h(\Omega),
\end{equation}
where ${a}^{DG}_h(u_h, v_h) $, defined on $X_h(\Omega)\times X_h(\Omega)$, and $f(v)$,
defined on $ X_h(\Omega)$, are given by
\begin{equation}\label{eq:def:a-h}
{a}^{DG}_h(u, v) =\sum^N_{i=1} {a}^{DG}_i(u,v)~~~
\mbox{and}~~~~f(v) = \sum_{i=1}^N \int_{\Omega_i} f v_i dx.
\end{equation}
Here each local bilinear form ${a}^{DG}_i$ is given as a sum of three symmetric 
bilinear forms:
\begin{equation}\label{eq:def:a^hat-i}
 {a}^{DG}_i(u,v) :=  a_i(u, v) + s_i(u, v) + p_i(u, v),
\end{equation}
where $a_i$ is the bilinear form associated with the ``energy",
\begin{equation}\label{eq:def:a-i}
a_i(u, v) := \int_{\Omega_i}\!\!\kappa(x) \nabla u_i  \cdot \nabla v_i dx,
\end{equation}
the $s_i$ is the bilinear form ensuring consistency and symmetry
\begin{equation}\label{eq:def:s-i}
s_i(u, v) := \sum_{E_{ij} \subset \partial\Omega_i}
 \frac{1}{l_{ij}}\int_{E_{ij}}  \kappa_{ij} \left( \frac{\partial u_i}{\partial n_i} (v_j - v_i)
+ \frac{\partial v_i}{\partial n_i} (u_j - u_i) \right) ds,
\end{equation}
and $p_i$ is the penalty bilinear form that is added for stability
\begin{equation} \label{eq:def:p-i}
p_i(u, v) := \sum_{E_{ij} \subset \partial \Omega_i}
\frac{1}{l_{ij}} \frac{\delta}{h_{ij}}
\int_{E_{ij}}
\kappa_{ij} (u_j - u_i)(v_j - v_i)ds.
\end{equation}
Here  $\kappa_{ij}$ 
is defined in (\ref{eq:def:averages}) and  $\frac{\partial}{\partial n_i}$
denotes the outward normal derivative on $\partial \Omega_i$.
The parameter $\delta$ is  a positive penalty parameter.
In order to simplify notation we included the
index $j=\bi$ in the definition
of the bilinear forms $s_i$ and $p_i$ above.
In order to include $E_{i\bi} := \partial \Omega_i \cap \partial \Omega$ in 
the summation sing, we set
 $l_{ij} = 2$ when $i,j\not = \partial$ 
and $l_{ij}=1$ when  $j=\partial$.
We also    let  $v_\bi = 0$
for all $v\in X_h(\Omega)$,  and define
$\kappa_{i\bi} = \kappa_i$ and $h_{i\bi} = h_i$.
We note that when $\kappa_{ij}$ is given by the harmonic average, then
 $\min\{\kappa_i,\kappa_j\} \leq \kappa_{ij} \leq 2\min\{\kappa_i,\kappa_j\}$.


For later use we define the positive bilinear forms $d_i$ as
\begin{equation}\label{eq:def:d-i}
d_i(u,v) = a_i(u,v) +  p_i(u,v),
\end{equation}
and the broken bilinear form $d_h$  for $X_h(\Omega)$:
\begin{equation}\label{eq:def:d-h}
d_h(u,v) := \sum_{i=1}^N d_i(u,v).
\end{equation}
For $u = \{u_i\}_{i=1}^N \in X_h(\Omega)$  the associated broken norm is then
defined by
 \begin{equation}\label{eq:def:broken-norm}
\|u \|_{h,\delta}^2=d_h(u,u) 
= \sum^N_{i=1} \left\{  \parallel \kappa_i^\frac{1}{2}\nabla u_i \parallel^2_{L^2(\Omega_i)} 
+ \sum_{E_{ij} \subset \partial \Omega_i}\!\!\!
\frac{1}{l_{ij}} \frac{\delta}{h_{ij}} \int_{E_{ij}} \kappa_{ij}(u_i - u_j)^2 ds\right\}.
\end{equation}

We also have the following lemma shown in \cite[Lemma 3.1]{Dryja_DG_03}.
Here we provide a sketch of the proof for the sake of completeness.

\begin{lemma} \label{lem:equiv}
There exists $\delta_0 > 0$ such that for $\delta \geq \delta_0$ and
for all  $u\in X_h(\Omega)$  the following inequalities hold:
\begin{equation*}
\gamma_0 d_i(u,u) \leq {a}^{DG}_i(u, u) \leq
\gamma_1 d_i(u,u),~~~~ i=1,\dots,N,
\end{equation*}
and
\begin{equation}\label{norm-quiv}
\gamma_0 d_h(u,u) \leq {a}^{DG}_h(u, u) \leq \gamma_1 d_h(u,u),
\end{equation}
where $\gamma_0$ and $\gamma_1$ are positive constants
independent of the
$\kappa_i$, $h_i$ $H_i$ and $u$.
\end{lemma}

\begin{proof}
First, we want to prove that
$ 
\gamma_0 d_i(u,u) \leq {a}^{DG}_i(u,u).
$ 
Since 
${a}^{DG}_i(u,u) = a_i(u,u)+ s_i(u,u)+d_i(u,u)$,
the proof reduces to bound  $|s_i(u,u)|$.   Note that
\[
s_i(u, u) = 2\sum_{E_{ij} \subset \partial\Omega_i}
 \frac{1}{l_{ij}}\int_{E_{ij}}  \kappa_{ij}  \frac{\partial u_i}{\partial n_i} (u_j - u_i)=
2\sum_{E_{ij} \subset \partial\Omega_i} \frac{1}{l_{ij}}I_{ij}
\]
where we have defined $I_{ij}:= \int_{E_{ij}} \kappa_{ij} \frac{\partial u_i}{\partial n_i}(u_j - u_i) ds$.
We have
\begin{eqnarray*}
I_{ij}   \leq   \left\|{ \kappa_{ij}^{1/2}}\nabla u_i \right\|_{L^2(E_{ij})} \left\| \kappa_{ij}^{1/2}(u_j -u_i)\right\|_{L^2(E_{ij})}.
\end{eqnarray*}

Using the following inequality for $ u_i \in X_h(\Omega_i)$
\begin{equation*}
h_i \left\| \kappa_{ij}^{1/2} \nabla u_i \right\|^2_{L^2(E_{ij})} \leq C \left\| \kappa_i^{1/2}\nabla u_i \right\|^2_{L^2(\Omega_i)},
\end{equation*}
the Young's inequality with arbitrary $\epsilon>0$ and the fact $h_{ij}\leq 2 h_i$ , we get
\begin{equation*}
I_{ij} \leq C \left\{\epsilon \left\|\kappa_{i}^{1/2}\nabla u_i\right\|^2_{L^2(\Omega_i)} + \frac{1}{4\epsilon} \frac{1}{2h_{ij}}\left\|\kappa_{ij}^{1/2} (u_j- u_i) \right\|^2_{L^2(E_{ij})}\right\}.
\end{equation*}

Then, multiplying  by ${1}/{l_{ij}}$ and 
summing over the edges $E_{ij} \subset \Omega_i$, we get 
\begin{eqnarray*}
|s_{i}(u,u)|&\leq& 2 C N_E  \epsilon \left\|\kappa_{i}^{1/2}\nabla u_i\right\|^2_{L^2(\Omega_i)}
+ \frac{C}{4\epsilon}\sum_{E_{ij} \subset \partial\Omega_i} \frac{1}{l_{ij}}
 \frac{1}{h_{ij}}\left\|\kappa_{ij}^{1/2} (u_j- u_i) \right\|^2_{L^2(E_{ij})}\\
&=&  2 C N_E \epsilon a_i(u,u)+\frac{C}{4\epsilon \delta} p_i(u,u).
\end{eqnarray*}
Here $N_E$ denotes the number of edges of subdomain $\Omega_i$. 
Choosing $\epsilon =1/(4 C N_E)$  we get 
\[
|s_i(u,u)|\leq 0.5 a_i(u,u)+ \frac{C^2N_E}{\delta} p_i(u,u)
\]
and then 
\begin{equation*}
\begin{split}
0.5a_i(u,u)+ 
(1-\frac{C^2N_E}{\delta}) p_i(u,u) \leq 
a^{DG}(u,u) \leq 1.5 a_i(u,u)+ (1+\frac{C^2N_E}{\delta})p_i(u,u).
\end{split}
\end{equation*}

Therefore the results holds if we take $\delta\geq\delta_0>CN_E$, 
$\gamma_0=\min\{0.5, 1-(C^2N_E)/\delta\}$ and $\gamma_1=\max\{1.5,(C^2N_E)/\delta\}$.
\end{proof}

\begin{remark}
We note that $\gamma_1/\gamma_0$ in Lemma \ref{lem:equiv}
deteriorates when $\delta$ gets larger. In practice, however,
$\delta \geq \delta_0$ is
chosen such that $\delta = O(1)$, therefore,
from now on we assume that all the estimates will not depend on $\delta$.
\end{remark}

\section{Coarse-grid spaces}\label{sec:spaces}


 In this section, we will construct local multiscale basis functions.
 We will follow GMsFEM where one needs the space of snapshots, see 
\cite{egh12,Review}.
In this space of snapshots, local spectral problems are designed
and solved to compute multiscale basis functions.
To keep our presentation simple, we first use the space
of snapshots to be fine-grid functions  within a coarse region.
Thus, the local spectral problems will be posed on the fine grid. 
Next, we will discuss how a general space of snapshots can be used.

\subsection{Fine-grid snapshot space and weighted eigenvalue problem}\label{sec:msI}

Following \cite{egw10,Review} we
consider the eigenvalue problem  in $\Omega_i$ 
for the eigenvalues $\lambda_{i, \ell}^{I}$
and the eigenfunctions $\psi_{i, \ell}^{I}(x)$:
\begin{equation}
\label{eq:eig:prob}
-\mbox{div}(\kappa(x) \nabla \psi_{i, \ell}^{I})=\lambda_{i, \ell}^{I} \widetilde{\kappa} \psi_{i, \ell}^{I},  ~~~x \in \Omega_i, \quad \kappa(x) \nabla \psi_{i, \ell}^{I} \cdot n=0,
~~~x \in \partial \Omega_i,
\end{equation}
where $n$ is the outer unit normal vector to $\partial \Omega_i$ and 
$\widetilde{\kappa}$ is a properly selected weight; for
scalar permeability, we select $\widetilde{\kappa}=\kappa$ while for 
tensor permeability we refer to \cite{egw10}.  The super-index $I$  is used 
to distinguish  from the other two methods we develop here (with indexes $II$ and $III$).

The eigenvalue problem considered above is solved 
in a discrete setting.  Namely, 
for any given subdomain $\Omega_i$ find $\psi_{i, \ell}^I \in X_h(\Omega_i)$ such that
\begin{equation}\label{eq:eigenvalueproblem}
a_i(\psi_{i, \ell}^{I}, z)=\lambda_{i, \ell}^{I}m_i(\psi_{i, \ell}^{I},z) \quad \mbox{ for all } z\in X_h(\Omega_i).
\end{equation}
Here, $a_i$ is defined in (\ref{eq:def:a-i}) and the bilinear form $m_i(\cdot,\cdot)$ is defined by
\begin{equation}\label{eq:def:m-i}
m_i(v,z)=\int_{\Omega_i} \kappa vz.
\end{equation}
We order the eigenvalues so that $ 0 \le \lambda_{i,1}^I\leq \lambda_{i,2}^I\leq ....\leq \lambda_{i,N_i}^I,$
where $N_i$ is the number of vertices of $\mathcal{T}_h(\Omega_i)$, i.e., the number of degrees of freedom associated to $X_h(\Omega_i)$.
 Then, in each subdomain $\Omega_i$,   we take the 
$L_i$ eigenfunctions corresponding to the smallest eigenvalues
 and use them as the multiscale basis. More precisely, define
\[
X_{H}^I(\Omega_i):=
\mbox{span}\left \{ \psi_{i,\ell}^I, 1\leq l \leq L_i\right \} \subset X_h(\Omega_i), \quad
i=1,\dots,N.
\] 

Finally, the coarse space is defined as 
\[
X_H^I(\Omega):= X_H^I(\Omega_1) \times \dots \times  X_H^I(\Omega_N) \subset X_h(\Omega).
\]
We refer to $X_H^I(\Omega)$ as an spectral coarse space due to its construction.

Now the coarse-grid  problem is: find $u^{ms,I}_H\in X_H^I(\Omega)$ such that
\begin{equation}\label{eq:disc-with-aDG}
{a}^{DG}_h(u^{ms,I}_H, v_H) = f(v_H) \quad \mbox{ for all }
\quad v_H\in X_H(\Omega).
\end{equation}
Note that the dimension of $ X_{H}^I(\Omega)$ depends on the 
number of eigenvectors chosen in each coarse block $\Omega_i$. An ideal situation would be 
when  small number of eigenvectors in $\Omega_i$ represent 
(approximate) well
the restriction of the solution to that subdomain.

\subsection{Fine-grid snapshot space with amended eigenvalue problem}\label{sec:msII}

Motivated by the error analysis developed below in Section \ref{sec:error},
we use the following modified eigenvalue problem. Find $\psi_{i,\ell}^{II} \in X_h(\Omega_i)$ such that
\begin{equation}\label{eq:ammendedeigenvalueproblem}
a_i(\psi_{i,\ell}^{II}, z)=\lambda_{i, \ell}\left( 
m_i(\psi_{i,\ell}^{II},z) + m^{\delta}_i(\psi_{i,\ell}^{II},z) \right) \quad \mbox{ for all } z\in X_h(\Omega_i)
\end{equation}
where $a_i$ is defined in (\ref{eq:def:a-i}), $m_i(\cdot,\cdot)$ is defined by (\ref{eq:def:m-i})
and 
\begin{equation}\label{eq:def:midelta}
m_i^{\delta}(v,z)=
\sum_{E_{ij} \subset \partial \Omega_i}
\frac{1}{l_{ij}} \frac{\delta}{h_{ij}}
\int_{E_{ij}}
\kappa_{ij} vzds.
\end{equation}
These eigenvalue problems allow us to obtain simple error estimates since the eigenvectors 
can approximate fine functions simultaneously in a norm that includes interior weighted semi-norm 
in a coarse-grid block and 
weighted $L^2$-norm on the interfaces. 

As before we order the eigenvalues as 
$0 \le \lambda_{i,1}^{II}\leq \lambda_{i,2}^{II}\leq \dots \leq \lambda_{i,N_i}^{II}$  and we choose $L_i$ eigenfunctions corresponding to the smallest eigenvalues
 and use them as the multiscale basis. Define
\[
X_{H}^{II}(\Omega_i):=
\mbox{span}\left \{ \psi_{i,\ell}^{II}, 1\leq l \leq L_i\right \} \subset X_h(\Omega_i), \quad
i=1,\dots,N,
\] 
and the coarse space
\[
X_H^{II}(\Omega):= X_H^{II}(\Omega_1) \times \dots \times  X_H^{II}(\Omega_N) \subset X_h(\Omega).
\]

Now the coarse-grid  problem similar to (\ref{eq:disc-with-aDG}): find $u^{ms,{II}}_H\in X_H^{II}(\Omega)$ such that
\begin{equation}\label{eq:disc-with-aDG2}
{a}^{DG}_h(u^{ms,{II}}_H, v_H) = f(v_H) \quad \mbox{ for all }
\quad v_H\in X_H^{II}(\Omega).
\end{equation}

\subsection{General snapshot space and an example}\label{sec:msIII}

In general, one can consider a general snapshot space
for solving local eigenvalue problems. As we discussed
in the Introduction, the use of general snapshot space can
have an advantage in case additional information is known 
about the local solution space. The subset of all possible function that satisfy 
the know properties of the unknown solution can be taken as the snapshot space. 
In this way we solve eigenvalue problem only on interesting (smaller dimension) 
subspaces instead of the space of fine degrees of freedom. For example,  if solutions need to be 
computed only for a subspace
of possible source terms,  one can restrict the snapshot space to the space
of local solutions for those sources  and do not consider all fine-grid functions. To demonstrate 
that it 
is possible to use a snapshot space strictly smaller than $X_h(\Omega)$,
we consider an example where the snapshot space 
 consists of  all local solutions of the homogeneous equation with
boundary conditions restriction on the boundary of the 
finite element nodal basis functions 
(or the set of all discrete $a_i-$harmonic 
functions in each block). 
%
%
More precisely, for the nodal basis function 
$\delta_k (x)$ corresponding to the 
$k-$th node on $\partial\Omega_i$,
we consider the problem 
\begin{equation}
\label{eq:eig:prob1}
-\mbox{div}(\kappa \nabla \phi_{i,k})=0  ~~~ \mbox{ in } \Omega_i, \quad  \phi_{i,k} =\delta_{k}
~~~\mbox{ on }  \partial \Omega_i,
\end{equation}
The $\phi_{i,k} \in X_h(\Omega_i)$, $k=1, \dots, M_i$, is
the (finite element) solution of this local problem. Here
$M_i$ denote the number of nodal basis function corresponding to nodes 
on $\partial \Omega_i$.
%
Then the space of snapshots is defined by
\begin{equation}\label{eq:def:snapshot}
X_h^{\mbox{snap}}(\Omega_i)=\mbox{span}\{ \phi_{i,k} , 1\leq k \leq M_i\}, \quad i=1, \dots, N.
\end{equation}

\begin{remark}\label{rem:referencesol}
Here, the reference solution we want to approximate on a coarse grid  is the Galerkin 
projection of $u^*_h$, solution of (\ref{eq:disc}), into the global snapshot space
$X_h^{\mbox{snap}}(\Omega)=
X_h^{\mbox{snap}}(\Omega_1)\times \dots \times X_h^{\mbox{snap}}(\Omega_N)$.
\end{remark}

Our objective is to construct  a possibly smaller dimension space $X_H^{\mbox{snap}}(\Omega_i)$ 
which is a subspace of  $X_h^{\mbox{snap}}(\Omega_i)$. The construction
is done using  appropriate spectral decomposition. For this, define the matrices 

\[
A_i^{\mbox{snap}} = [ a_i (\phi_{i,k}, \phi_{i,k'}) ]_{k,k'=1,}^{M_i} \mbox{ and }
M_i^{\mbox{snap}} = [ m^{\delta}_i (\phi_{i,k}, \phi_{i,k'}) ]_{k,k'}^{M_i}
\]
and solve  the following algebraic eigenvalue problem 
\begin{equation}
\label{eq:eig11}
A_i^{\mbox{snap}} \alpha_{i,\ell}=\lambda_{i,\ell}^{\mbox{snap}} M_i^{\mbox{snap}} \alpha_{i,\ell}.
\end{equation}
We write $\alpha_{i,\ell}=(\alpha_{i,\ell; 1},\dots,\alpha_{i,\ell;M_i})\in \mathbb{R}^{M_i}$ 
and define the corresponding finite element functions, $\psi_{i,\ell}^{III}\in X_h(\Omega_i)$ as
\[
\psi_{i,\ell}^{III}=\sum_{k=1}^{M_i}\alpha_{i,\ell;k} \phi_{i,k}, \quad \ell=1,\dots, M_i.
\]

Note that the matrices $A_i^{\mbox{snap}} $ and $M_i^{\mbox{snap}} $ are
computed in the space of snapshots in $\Omega_i$.
Assume that 
$ 0 \le \lambda_{i,1}^{\mbox{snap}}\leq ...\leq 
\lambda_{i,M_i}^{\mbox{snap}},
$
and choose the  $L_i$ eigenvectors $\psi_{i,1}^{III}, \dots, \psi_{i,L_i}^{III} $
that correspond to {the smallest}  $L_i$ eigenvalues.
We intoduce 
\[
X_H^{III}(\Omega_i)=\mbox{span}\left \{\psi_{i,\ell}^{III} : l=1, \dots, L_i\right \}  \mbox{ for } i=1, \dots, N
\]
and define the global coarse space by
\[
X_H^{III}(\Omega):= X_H^{III}(\Omega_1) \times \dots \times  X_H^{III}
(\Omega_N) \subset X_h(\Omega).
\]
The coarse problem is: find $u^{ms,III}_H\in X_H^{III}(\Omega)$ such that
\begin{equation}\label{eq:disc-with-aDG3}
{a}^{DG}_h(u^{ms,III}_H, v_H) = f(v_H) \quad \mbox{ for all }
\quad v_H\in X_H^{III}(\Omega).
\end{equation}

\begin{remark}
Note that, according to the definition of $m_i^\delta$ in (\ref{eq:def:midelta}), 
the matrix $M_i^{\mbox{snap}} $ scales with $1/h_{ij}$. 
Then, the resulting eigenvalues scale  with $h_{ij}$ while 
the eigenspaces do not depend on $h_{ij}$. 
A similar situation is  also valid for Method II and the 
eigenvalue problem (\ref{eq:ammendedeigenvalueproblem}).  
It is easy to see from our main Theorem \ref{thm:error} (estated and proved below) 
that this scaling does not affect the convergence rate with respect to the 
number of eigenvectors used in the coarse space.
\end{remark}

\begin{remark}
Instead of the $M_i^{\mbox{snap}} $ defined above, we 
can use 
\[
M_i^{\mbox{snap}} = [
 m_i (\phi_{i,k}, \phi_{i,k'})+
 m^{\delta}_i (\phi_{i,k}, \phi_{i,k'}) ]_{k,k'=1}^{M_i}.
\]
\end{remark}

\subsection{Stability estimate}

In this section, we present a best approximation  result for the coarse-grid
solution.

\begin{lemma}\label{lem:bestapprox}
Let $u^*_h\in X_h(\Omega)$ and $u_H^{ms,I}\in X_H^{I}(\Omega)$ be the solutions of 
(\ref{eq:diff}) and (\ref{eq:disc-with-aDG}), correspondingly. 
We have
\begin{equation}
d_h(u^*_h-u_H^{ms,I}, u^*_h-u_H^{ms,I}) \leq C_1 d_h(u^*_h-v, u^*_h-v)~~ \mbox{for~ all}~~ v \in 
X_H^I(\Omega)
\end{equation}
with $C_1$ is independent of $\kappa_i, h_i, H_i, u^*_h$ and $u_H^{ms,I}.$
\end{lemma}
\begin{proof}
For all $v \in X_H^{I}(\Omega),$
\begin{equation}
a_h^{DG}(u^*_h, v)=f(v) \quad  \mbox{ and } \quad a_h^{DG}(u_H^{ms,I}, v)=f(v).
\end{equation}
Then $a_h^{DG}(u^*_h-u_H^{ms,I}, v)=0$ and since $u_H^{ms,I} \in X_H^I(\Omega),$
\begin{equation}
a_h^{DG}(u^*_h-u_H^{ms,I}, u^*_h-u_H^{ms,I}) = a_h^{DG}(u^*_h-u_H^{ms,I}, u^*_h-v).
\end{equation}
Using a Cauchy-Schwarz inequality and (\ref{norm-quiv}) in Lemma \ref{lem:equiv},
\begin{equation*}
\begin{split}
\gamma_0 d_h(u^*_h-u_H^{ms,I}, u^*_h-u_H^{ms,I}) & \leq a_h^{DG}(u^*_h-u_H^{ms,I}, 
u^*_h-u_H^{ms,I}) = a_h^{DG}(u^*_h-u_H^{ms,I}, u^*_h-v) \\
 & \leq a_h^{DG}(u^*_h-u_H^{ms,I}, u^*_h-u_H^{ms,I})^{1/2} a_h^{DG}(u^*_h-v, u^*_h-v)^{1/2}\\
 & \leq \gamma_1 d_h(u^*_h-u_H^{ms,I}, u^*_h-u_H^{ms,I})^{1/2} d_h(u^*_h-v, u^*_h-v)^{1/2}.
\end{split}
\end{equation*}
Taking $C_1 = ({\gamma_1}/{\gamma_0})^2,$ we get
\begin{equation*}
d_h(u^*_h-u_H^{ms,I}, u^*_h-u_H^{ms,I}) \leq C_1 d_h(u^*_h-v, u^*_h-v)
\end{equation*}
and this completes the proof.
\end{proof}

Analogous best approximation results, with respect to the fine-grid reference solutions,  hold true for the other two coarse problems 
(\ref{eq:disc-with-aDG2}) and (\ref{eq:disc-with-aDG3}).

\subsection{Error estimates in terms of the local energy captured}\label{sec:error}
The following theorem gives and error estimates with respect to the number of eigenvectors used. Therefore, 
we obtain convergence to the reference solution when we add more and more eigenvectors to
the coarse space. The error estimates are written in terms of the amount of local energy (of the reference solution)
that is captured using the selected eigenmodes.

\begin{theorem}\label{thm:error}
Let $u^*_h$ and $u_H^{ms,II}$ be the solution of 
Problem  (\ref{eq:disc}) and (\ref{eq:disc-with-aDG2}), respectively. 
Put $u^*_h=\{ u^*_i\}_{i=1}^N$. 
In each subdomain we can write 
\[
u^*_i=\sum_{\ell=1}^{N_i} c_{\ell}(u^*_i) \psi_{i,\ell}^{II} \quad \mbox{ and }\quad 
a_i(u^*_i,u^*_i)=\sum_{\ell=1}^{N_i} \lambda_{i,\ell}^{II} c_\ell(u^*_i)^2,
\]
where $ c_\ell(u^*_i)=m_i(u^*_i,\psi_{i,\ell}^{II})+m_i^\delta(u^*_i,\psi_{i,\ell}^{II})$. 
The following error estimate  holds 
\[ \displaystyle
d_h(u^*_h-u_H^{ms,II}, u^*_h-u_H^{ms,II}) \leq C_1\left(1+\frac{4}{ \displaystyle\min_{1\leq i\leq N}\lambda_{i,L_i+1}^{II}}\right) 
\sum_{i=1}^N \sum_{\ell=L_i+1}^{N_i} \lambda_{i,\ell}^{II} c_\ell(u^*_i)^2.
\]

\end{theorem}
\begin{proof}
Using the truncated expansion of solutions, define the interpolation $I^H(u^*_h)$ 
 by 
\[
I^H(u^*_h)=\{ I^H_i(u^*_i)\}_{i=1}^N \mbox{ where }  I^H_i(u^*_i)=
\sum_{\ell=1}^{L_i} c_\ell(u^*_i) \psi_{i,\ell}^{II}
\]
where $ c_\ell(u^*_i)=m_i(u^*_i,\psi_{i,\ell}^{II})+m_i^\delta(u^*_i,\psi_{i,\ell}^{II})$, 
$i=1,\dots,N$. Note that  
$ I^H_i(u^*_i)$ is the projection of $u^*_i$ into the space spanned by the first $L_i$ eigenvectors.
Now we take $v=I^H(u^*_h)$ in Lemma \ref{lem:bestapprox} to obtain
\begin{equation}\label{thm:1}
d_h(u^*_h-u_H^{ms,II}, u^*_h-u_H^{ms,II})  \leq C_1 d_h(u^*_h-I^H(u^*_h), u^*_h-I^H(u^*_h))=d_h(e, e)
\end{equation}
where we have defined $e=\{e_i\}_{i=1}^N$ where
\[
e_i=u^*_i-I^H_i(u^*_i)=\sum_{\ell=L_i+1}^{N_i} c_\ell(u^*_i)  \psi_{i,\ell}^{II}.
\]
Now we bound $d_h(e,e)$. First we observe that 
\begin{equation}\label{thm:2}
d_h(e,e)=\sum_{i=1}^N  \big( a_i(e_i,e_i)+ p_i(e,e)\big).
\end{equation}
The second term in this sum can be bounded as follows. We write
\begin{eqnarray*}
p_i(e,e)&=&\sum_{E_{ij} \subset \delta \Omega_i}
\frac{1}{l_{ij}} \frac{\delta}{h_{ij}}
\int_{E_{ij}}
\kappa_{ij} (e_j - e_i )^2ds\\
&\leq & 2\sum_{E_{ij}  \subset \partial \Omega_i} \left( \frac{1}{l_{ij}} \frac{\delta}{h_{ij}}
\int_{E_{ij}}
\kappa_{ij} (e_j )^2ds +
\frac{1}{l_{ij}} \frac{\delta}{h_{ij}}
\int_{E_{ij}}
\kappa_{ij} (e_i )^2ds\right)\\
&\leq & 2\sum_{E_{ij}  \subset \partial \Omega_i}  \frac{1}{l_{ij}} \frac{\delta}{h_{ij}}
\int_{E_{ij}}
\kappa_{ij} (e_j )^2ds +2m_i^\delta (e_i,e_i).
\end{eqnarray*}
Adding over all subdomains we get
\begin{eqnarray}
\sum_{i=1}^Np_i(e,e)&\leq & 
2 \sum_{i=1}^N \sum_{E_{ij}  \subset \partial \Omega_i}  \frac{1}{l_{ij}} \frac{\delta}{h_{ij}}
\int_{E_{ij}}
\kappa_{ij} (e_i )^2ds +2\sum_{i=1}^Nm_i^\delta (e_i,e_i) \nonumber\\
&=& 4\sum_{i=1}^Nm_i^\delta (e_i,e_i). \label{thm:3}
\end{eqnarray}
On the other hand, if we use the increasing order of eigenvalues of eigenvalue problem (\ref{eq:ammendedeigenvalueproblem}), 
we get,
\begin{equation}\label{thm:4}
m_i^\delta (e_i,e_i)\leq m_i (e_i,e_i)+m_i^\delta (e_i,e_i) \leq \frac{1}{\lambda_{i,L_i+1}^{II}} a_i(e_i,e_i)
\end{equation}
which, together with (\ref{thm:2}) and (\ref{thm:3}), gives 
 \begin{eqnarray*}
\sum_{i=1}^N \big( a_i(e_i,e_i)+p_i(e,e) \big) &\leq & \sum_{i=1}^N  (a_i(e_i,e_i)+4m_i^\delta(e_i,e_i)\big) \\
&=&  \sum_{i=1}^N \left(1+\frac{4}{\lambda_{i, L_i+1}^{II}}\right)
\sum_{\ell=L_i+1}^{N_i} \lambda_{i,\ell}^{II} c_\ell(u^*_i)^2\\
&\leq & \left(1+\frac{4}{ \displaystyle \min_{1\leq i\leq N}\lambda_{i,L_i+1}^{II}}\right) 
\sum_{i=1}^N
\sum_{\ell=L_i+1}^{N_i} \lambda_{i,\ell}^{II} c_\ell(u^*_i)^2.
\end{eqnarray*}
%
This completes the proof.
\end{proof}
\begin{remark}
Using our analysis, in order to obtain further bounds for the error we have to study 
the convergence 
of the sum $\sum_{\ell=1}^{N_i} \lambda_{i,\ell}^{II} c_\ell(u^*_i)^2$ (that is, 
the decay  of the coefficients $c_\ell(u^*_i)^2$ with increasing $\ell$). This
can depend on the smoothness of the solution and it will be matter of 
further research.
\end{remark}

\begin{remark}
A similar result holds for the method constructed with snapshot 
space presented in Subsection \ref{sec:msIII}. In this case the reference 
solution is the solution obtaining by a Galerkin projection on the snapshot space. 
See Remark \ref{rem:referencesol}.
\end{remark}

\section{Numerical experiments}\label{sec:numerics}

In this section we present representative numerical experiments. 
In particular, we compute 
the coarse (or upscaled) solution  and study the error with respect to the
reference solution (or the fine-grid solution of (\ref{eq:disc})).  We choose $\delta=4$ 
for all the numerical test presented here.
We note that the solution of (\ref{eq:disc}) depends on both 
fine-scale and coarse-scale parameters, $h$ and $H$. We 
are interested mainly on the convergence  (to the reference solution)
when we sequentially add more and more basis functions. 
We study the error behavior due to the addition of coarse basis 
functions for fixed value of $h$ and $H$. 

We consider the domain $\Omega = (0,1)^2$ and divide $\Omega$ into $N=M\times M$ square coarse 
blocks, $\{\Omega_i\}_{i=1}^{N}$, which are unions of fine elements. 
In this case  $H=1/M$ is the coarse mesh parameter.  Inside each subdomain 
$\Omega_i$ we generate a structured triangulation with $m$ subintervals in each coordinate 
direction (and thus  $h=1/(Mm)$ is the fine mesh parameter).
We consider the solution of Equation (\ref{eq:disc}) with $f=1$ and a high contrast coefficient 
described in  Figure \ref{fig:coef-perm2_2}.
This coefficient is one in the white background and 
value $\eta$ in the gray regions representing high-contrast channels and 
high-contrast inclusions. Thus, $\eta$ represents the contrast of the media, namely the ratio
of the maximum and minimum values of $\kappa(x)$.
\begin{figure}[H]
\centering
\includegraphics[width=5cm, height=5cm]{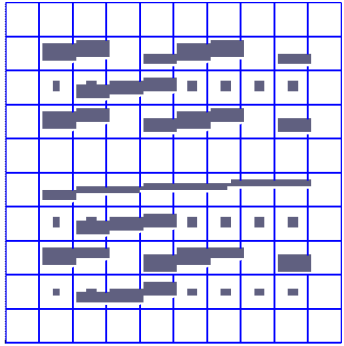}
\caption{High contrast coefficient.}
\label{fig:coef-perm2_2}
\end{figure}

In the following we compute the norm of the error  
$e=u^*_{h}-u_{H}$ between the 
fine-scale solution obtained by solving (\ref{eq:disc}) and the coarse-grid solution $u_{H}$,
which is one of the following coarse-grid function:
1. $u_H^{ms,I}$
solution of \eqref{eq:disc-with-aDG},  
2. $u_H^{ms,II}$ the solution of  \eqref{eq:disc-with-aDG2}, 
or
3.  $u_H^{ms,III}$ the solution of  \eqref{eq:disc-with-aDG3}. 
The  total error is 
$\|e \|_{h,1}^2$ where $\|\cdot\|_{h,\delta}^2$ defined in (\ref{eq:def:broken-norm}). 
The relative error is computed as $\|e\|_{h,1}^2/\|u^*_{h}\|_{h,1}^2$.
The error is divided into two quantities: 
\begin{itemize}
\item  \emph{Interior Error}:  (square of the) broken $H^1-$semi-norm of the error
\[
\sum_{i=1}^N a_i(e,e)= \sum^N_{i=1}   \parallel \kappa_i^\frac{1}{2}\nabla e_i \parallel^2_{L^2(\Omega_i)}.
\]
\item \emph{Interface Error}: (square of the) $L^2-$norm of the jump of the error across the edges
\[
\sum_{i=1}^N\sum_{E_{ij} \subset \partial \Omega_i}\!\!\!
\frac{1}{l_{ij}} \frac{1}{h_{ij}} \int_{E_{ij}} \kappa_{ij}(e_i - e_j)^2. 
\]
\item {\emph{Energy error}}: (square of the) DG bilinear form, that is, $a_h^{DG}(e,e)$.
\end{itemize}

\subsection{Fine-grid snapshot space and original eigenvalue problem}\label{sec:fg1}
In this Subsection we present the numerical experiments for  the method introduced  
in Subsection \ref{sec:msI} and show the error obtained when the dimension of the coarse 
space is increased.

First, we recall that for high-contrast problems we include the eigenvectors 
corresponding to small eigenvalues  (that asymptotically vanish 
as the contrast increases). 
We  denote by $L_{i}^{small}$ the number of these small eigenvalues in $\Omega_i$. 
To see the effect of adding more basis functions, we select additional $L_i^{add}$
eigenvalues so the total number of eigenvalues selected in the block $\Omega_i$ is 
$L^{small}_i+L^{add}_i$. We show that the error decays as $L_i^{add}$ increases.
For the coefficient $\kappa(x)$ and  coarse mesh shown 
on Figure \ref{fig:coef-perm2_2} there is only one such and eigenvalue 
in each coarse-grid block $\Omega_i$
and therefore
$L_i^{small}=1$, $i=1,2,\dots,N$. 
Figure \ref{fig:solution-perm2_2}  illustrates the effect of using increasing number 
of eigenvectors in the solution. 
We show the fine-scale solution and  coarse-scale solutions computed with three different 
coarse spaces
$L_i^{add}=0,2,11$. 
\begin{figure}[htb]
\centering
\includegraphics[width=13cm, height=9cm]{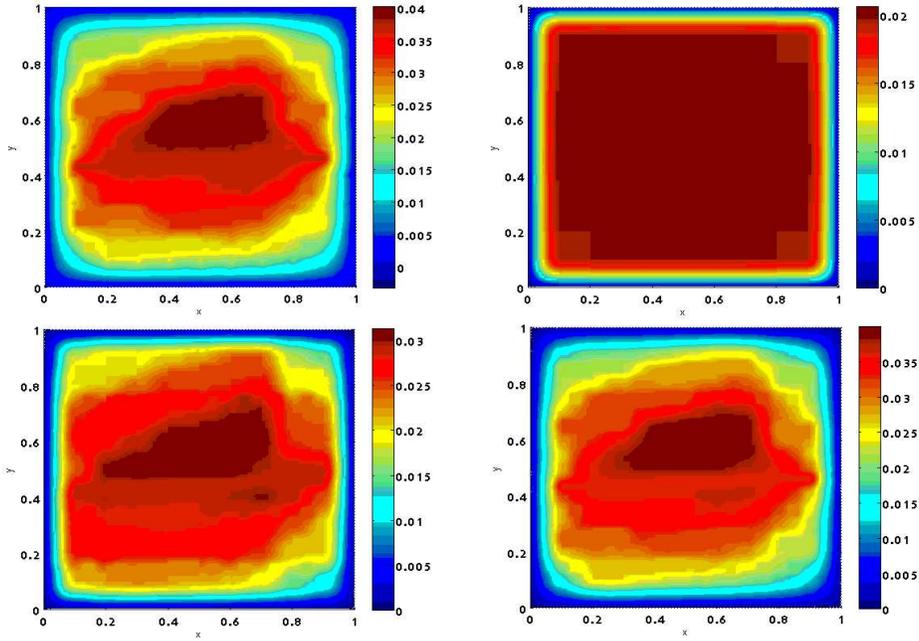}
\caption{Illustration of fine and coarse solutions. Fine-scale solution $u^*_h$ (top left).
Coarse-scale solution $u_H^{ms,I}$ with $L_i^{add}=0$ eigenvalues (top right).
Coarse-scale solution $u_H^{ms,I}$ with $L_i^{add}=3$ eigenvalues (bottom left).
Coarse-scale solution $u_H^{ms,I}$ with $L_i^{add}=11$ eigenvalues (bottom right).}
\label{fig:solution-perm2_2}
\end{figure}
\begin{figure}[htb]
{\includegraphics[width=10cm, height=5cm]{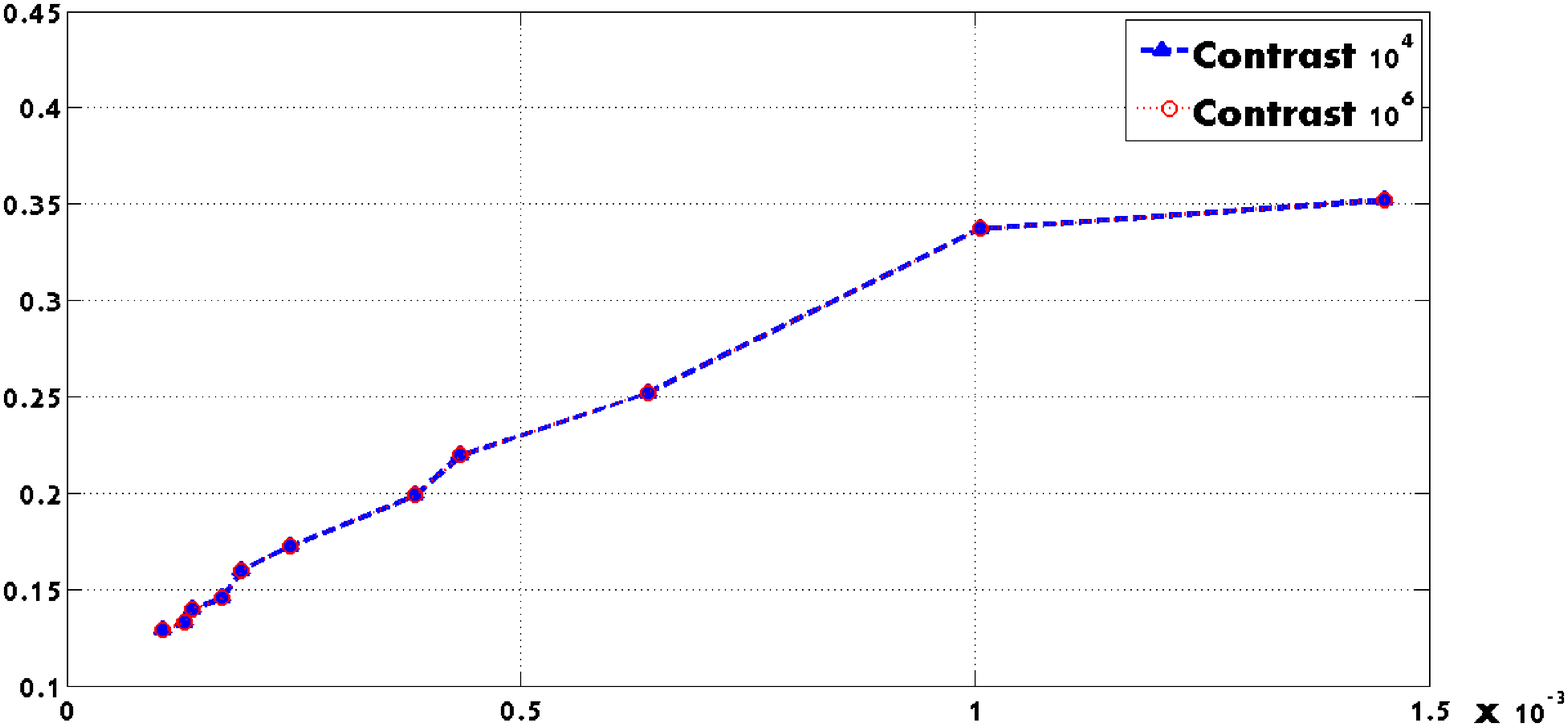}}
\caption{Total error of Method I (in Section \ref{sec:msI}) vs. $1 / \lambda_{min}$. Here $h=1/100$ and the contrast is $10^6$. 
Coarse mesh size  $H=1/10$ and contrast $10^4$ and $10^6$.}
\label{lambdaminplot}
\end{figure}

\begin{table}[h]\footnotesize
\caption{Numerical results for the Method I (see Section \ref{sec:msI}). Here, $h=1/100$, $H=1/10$, $\mu = 10^4$ and $\mu=10^6$ (in parenthesis)} 
\centering 
\begin{tabular}{r r r r r r r r}
\hline
$L_i^{add}$& Dim. & Interface error & Interior error & Total error&$\lambda_{\min}$ \\ [1ex]
\hline
0&100 (100)&         0.026 (0.026)&       0.326 (0.326)&      0.3522 (0.3522)&      689.4 ( 689.3)\\
2&300 (300)&	0.031 (0.032)&	0.221 (0.220)&	0.2523 (0.2516)&	1562.2 (1561.7)\\
4&500 (500)&	0.028 (0.029)&	0.171 (0.170)&	0.1991 (0.1984)&	2607.5 (2607.0)\\
6&700 (700)&	0.027 (0.027)&	0.133 (0.131)&	0.1600 (0.1581)&	5199.4 (5199.3)\\
8&900 (900)&	0.026 (0.027)&	0.114 (0.113)&	0.1399 (0.1392)&	7237.9 (7237.6)\\
10&1100 (1100)&	0.025 (0.025)&	0.104 (0.103)&	0.1293 (0.1283)&	9509.1 (9509.0)\\ [1ex] 
\hline 
\end{tabular}
\label{table:errorforlambda_10^4_II} 
\end{table}

The results for the computation of interior and interface errors are presented in 
Table \ref{table:errorforlambda_10^4_II} for $h=1/100$ and $H=1/10$
and two different contrasts.   
The convergence with respect to the 
minimum left out eigenvalue is shown in Figure \ref{lambdaminplot} (left). 
In this case, we solve a $100\times 100$ eigenvalue problem in each coarse block.
From the results we see  convergence to the reference solution (fine-grid solution). We also
observe error decay proportional to the minimum left out  eigenvalue across all coarse blocks. 
In particular,   for $H=1/10$, we need only 3 or 4 additional functions 
to get an interior error or the order of $17\%$. This error is computed with
respect to the fine-grid solution with fine-grid parameter $h=1/100$. Note that, in this case, we have total of
four  or five basis functions per subdomain which is comparable to the number of degrees of freedom 
of a classical DG method on the coarse grid.

\subsection{Error vs. coarse problem penalty scaling}
Now we test the error when we change the scaling of the penalty. 

We recall that the fine-scale problem in (\ref{eq:disc})  uses 
a penalty term scaled by $\delta\frac{1}{l_{ij}}\frac{1}{h_{ij}}$.
In the classical SIPG formulation  on the coarse grid one uses a penalty  scaled by $\delta\frac{1}{l_{ij}}\frac{1}{H}.$ 
Here we experiment by computing the coarse solutions with several penalties in the range from $1/H$ to $1/h$ to 
identify a good penalty parameter for the coarse problem.
\begin{figure}[htb]
{\includegraphics[width=10cm, height=7cm]{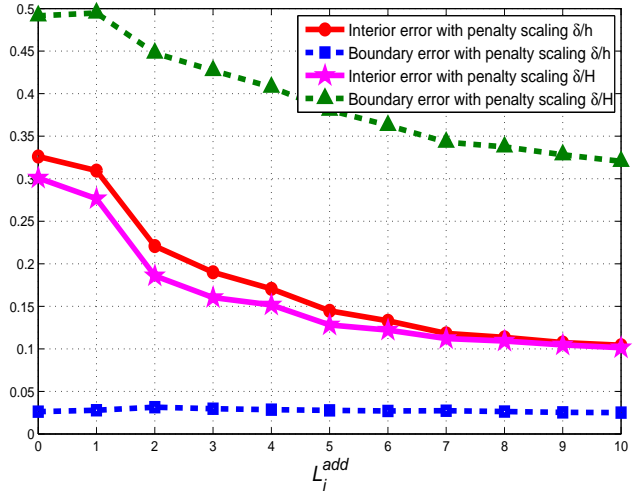}}
\caption{The error for values of $L_i^{add}=0,1,\dots,10$  and  the 
two different penalty scalings: $\frac{\delta}{h}$ and $\frac{\delta}{H}.$ 
Here $h=1/100$, $H=1/10$ and  $\eta=10^4$. }
\label{penalty}
\end{figure}

\begin{figure}[htb]
{\includegraphics[width=7.6cm, height=4.5cm]{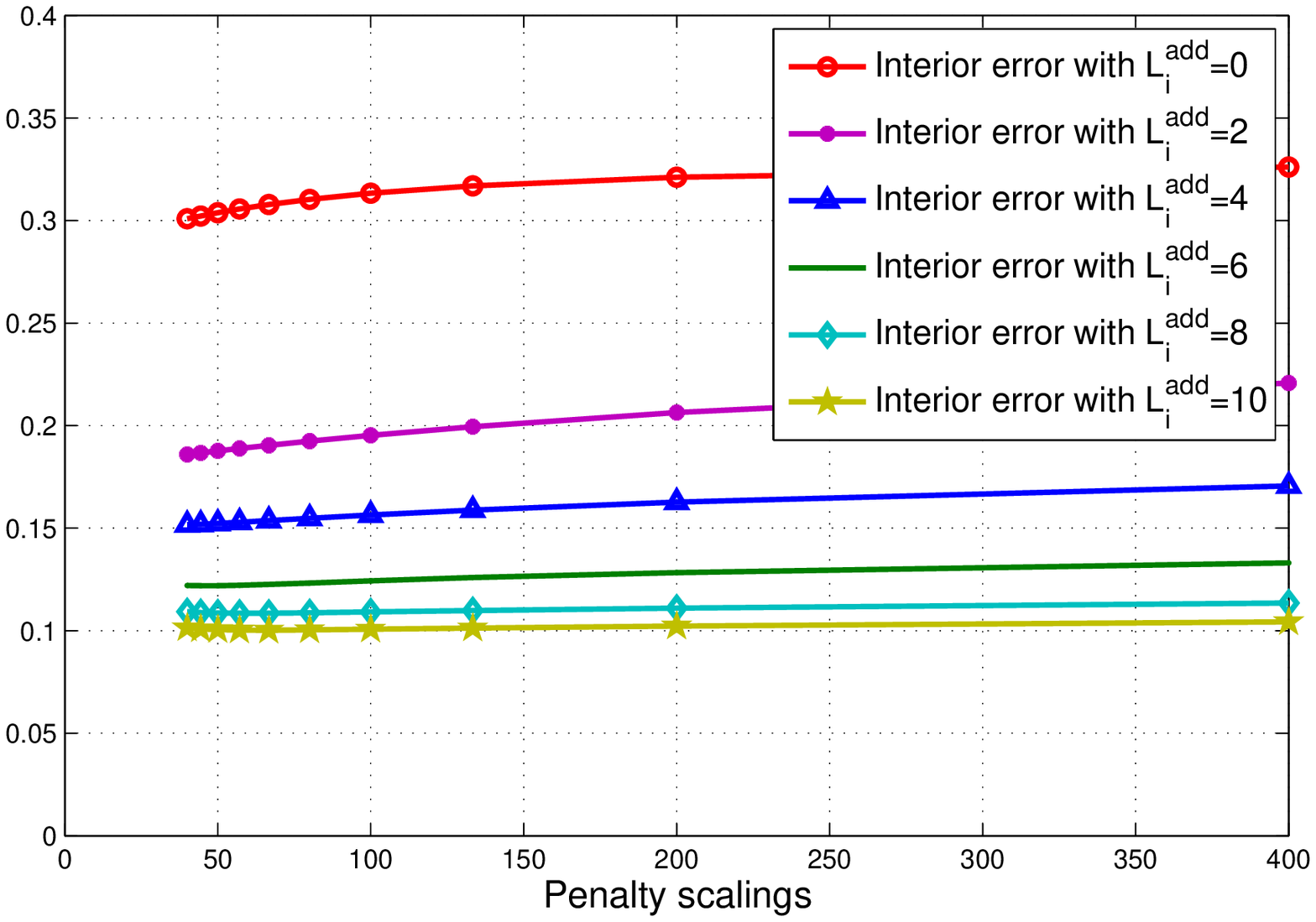}}
{\includegraphics[width=7.6cm, height=4.5cm]{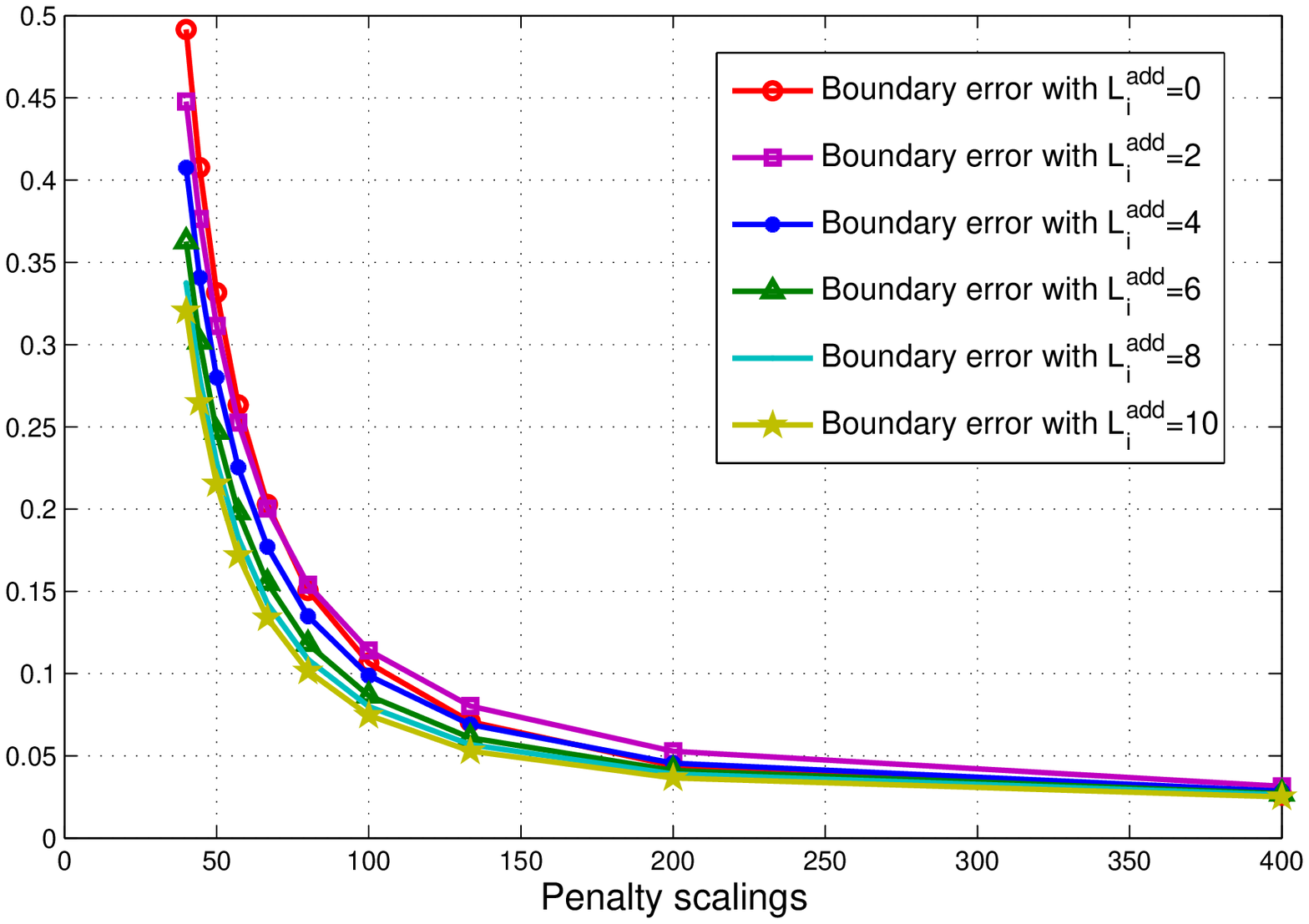}}\\
\caption{The error for Method I (in Section \ref{sec:msI} vs penalty scaling
for different values of $L_i^{add}$ . Here $h=1/100,$ $H=1/10$ and $\eta=10^4$. 
The interior error (left) and the boundary error (right).}
\label{error}
\end{figure}

For this experiment we set  the  contrast 
$\eta=10^4$, $M=10$, $n_i=10$, $i=1,\dots,N$  (and thus $H = 1/10$ and $h = 1/100$). 
Then,  recalling that $\delta=4$ for the numerical experiments, we have $\delta\frac{1}{H}=40$ and 
$\delta\frac{1}{h_{ij}}=400$.
For these two choices of the penalty in Figure \ref{penalty} we show the decay 
of the  interior and interface error when adding more eigenfunctions. We observe a reduction of the error as we use more and more additional coarse-grid  
basis functions. Also, we observe that the interior error are of comparable size (with either scaling) 
and that the interface error is bigger if we use the coarse penalty scaling.

We fix the number of additional eigenvectors $L_i^{add}$ and compute the interior and boundary errors when the coarse solution 
is computed with different penalties. We note that the fine-scale solution is computed with the fine-grid scaling of the penalty bilinear form.
We repeat this experiments with $L_i^{add}=0,1,\dots,10$. Results are shown on Figure \ref{error}.
From these results we observe that the boundary error for the coarse-grid solution
is more sensitive to the variations in the scaling of the penalty.
Indeed,
for example $M=10$, with $L_i^{add}=4$, $i=1,\dots,N$, the optimal penalty coefficient
is approximately  $70$.  These set of experiments demonstrates that one needs to properly choose
the penalty parameter in order to balance boundary and interior
errors.

\subsection{Fine-grid snapshot space and amended eigenvalue problem}
Here we consider the method introduced in Subsection \ref{sec:msII}. 
We repeat the experiment described in Subsection \ref{sec:fg1}. 
\begin{table}[ht]\footnotesize
\caption{Numerical results for Method II (Subsection \ref{sec:msII}) with 
 increasing dimension of the coarse space, $h=1/100$, 
$H=1/10$,  and  $\eta = 10^4$.} 
\centering 
\begin{tabular}{r r r r r r r r}
$L_i^{add}$ & Dim & Interface& Interior & Total& Energy&$\lambda_{\min}$ \\ [1ex]
\hline
0& 100		&0.0318	& 0.2854	& 0.3172	&  0.3172& 0.0528\\
2&  300		&0.0342	& 0.1679	& 0.2020	& 0.2646& 0.0933\\
4&  500		&0.0257	& 0.1066	& 0.1323	& 0.1793 &0.1165\\
6&  700		&0.0214	& 0.0800	& 0.1014	& 0.1444 & 0.2459\\
8&  900		&0.0194	& 0.0581	& 0.0775	& 0.1164&0.3514\\
10& 1100  	&0.0185	& 0.0566	& 0.0751& 0.1132 & 0.4551\\[1ex] 
\hline 
\end{tabular}
\label{table:errorforlambda_amended} 
\end{table}
The results are displayed in Table \ref{table:errorforlambda_amended} for 
contrast $\eta=10^4$. Similar results where observed for higher contrast. We observe error reduction when  
next eigenfunctions are added. Note that the results obtained by using the amended eigenvalue 
problem in Subsection \ref{sec:msII} are slightly better. In this case our numerical results 
verify our theoretical error estimates in Theorem \ref{thm:error}.  Note that we report the energy error.

\subsection{Local solutions as snapshot space}
Next, we consider the snapshot space given as in Subsection 
\ref{sec:msIII} that consists of 
$a_i-$harmonic functions defined in (\ref{eq:eig:prob1}).
We note that this space of snapshots are used in the generalized
multiscale finite element method for wave equation in \cite{eric-2012}.
The objective of presenting these results is to show that our
proposed DG method is flexible and one can use various snapshot
spaces. 
\begin{table}[ht]\footnotesize
\caption{Numerical results for Method III (see Subsection \ref{sec:msIII}) 
with increasing dimension of the coarse space, $h=1/100$, $H=1/10$ and $\eta = 10^4$.} 
\centering 
\begin{tabular}{r r r r r r r r}
\hline
$L_i^{add}$ & Dim & Interface & Interior& Total& Energy&$\lambda_{\min}$ \\ [1ex]
\hline
0& 100		&0.032	&0.285	& 0.317       &0.373&  0.0528\\
2&  300		&0.034	& 0.168	& 0.202	&0.265&  0.0933\\
4&  500		&0.026	& 0.107	& 0.133	&0.181&0.1165\\
6&  700		&0.021	& 0.079	& 0.101	&0.144& 0.2459\\
8&  900		&0.019	& 0.058	& 0.077	&0.116& 0.3514\\
10 & 1100     &0.018	& 0.057	& 0.075	&0.114& 0.4552\\[1ex] 
\hline 
\end{tabular}
\label{table:errormsIII} 
\end{table}

In  Table \ref{table:errormsIII}, we present the numerical results with contrast $\eta=10^4$. Note, that
we have obtained similar results for  contrast $\eta=10^6$ (not reported here).
In this example, we choose $H=1/10$ and the same setup as in the  previous section. 
From these results we observe convergence to the fine-grid solution of (\ref{eq:disc}). We also 
observe that the 
error is inversely proportional to the $\lambda_{min}$ (the minimum left out eigenvalue). Note that, we are reporting the error not 
with respect to the reference solution (on the snapshot space) but with respect to the fine-grid solution. Nevertheless, 
we observe good results.

We recall that we need to use the bilinear form $m^\delta_i$ in the eigenvalue problem in order to obtain error estimates. 
We also observe convergence in the numerical tests if we use $m_i$ bilinear form (instead of $m^\delta_i$) in the snapshot 
space of harmonic functions.  These  results are reported  in Table \ref{table:example2_10^4_1}.

\begin{table}[h]\footnotesize
\caption{Numerical results for snapshot space and bilinear form $m_i$ (instead of $m^\delta_i$). Here $h=1/100$, $H=1/10$, and   $\eta=10^4$.} 
\centering 
\begin{tabular}{r r r r r }
\hline   
$L_i^{add}$ & Dim. & Interface & Interior & Total \\\hline
  2 & 300&   	0.0301  &	0.1750 &  0.2051\\
  4 & 500&  	0.0278  &	0.1104 &  0.1382\\
  6 & 700&  	0.0257  &	0.0905 &  0.1162\\
  8 &  900&  	0.0243  &	0.0763 &  0.1007\\
 10&1100&    0.0223 &      0.0663 &  0.0986\\ [1ex] 
\hline 
\end{tabular}
\label{table:example2_10^4_1} 
\end{table}

\section{Discussions on the convergence }\label{sec:discuss}

In this section, we discuss the convergence of the proposed 
discontinuous multiscale finite element method. 
For the multiscale method in Section \ref{sec:msII},  we derived error estimates 
 in Theorem \ref{thm:error}. Similar result holds for the 
coarse space described in Section \ref{sec:msIII}. 
Despite of the fact we do not write an error estimate for the multiscale space in 
Section \ref{sec:msI}, which uses simple weighted eigenvalue problems, we verified convergence 
in the numerical experiments. 
Moreover,  we observe that the interface error due to penalty term
is smaller than the interior error calculated with the energy norm
(see Table \ref{lambdaminplot}). This indicates
that once we have sufficient number of multiscale basis functions per coarse
region (that are selected properly with our spectral problem) the interface error 
is dominated by the interior error.
Consequently, one can argue that the local basis function construction
should mostly take into account the approximation property within 
coarse regions the idea that we follow in this paper.
The proposed eigenvalue problem attempts eliminating some interior 
degrees of freedom via an (interior and a boundary) mass matrix as  
discussed in our previous papers, e.g. \cite{Review,egw10}. 
By selecting mass matrix carefully in
the local spectral problem, we represent piecewise constant functions
within high-conductivity inclusions as the exact solution becomes almost 
constant within these regions.
Note that without the $\kappa$ (or $\kappa_{ij}$) weight in the mass bilinear form, we will be selecting
large number of important modes 
(all fine-grid degrees of freedom within high-conductivity regions according to previous studies).


We have amended the local spectral problem  with a  mass term 
on the boundary (see Section \ref{sec:msII} and \ref{sec:msIII}) in 
order to obtain error estimates in terms of the 
local energy captured by the local space. 
A number of other similar eigenvalue problems that we tried, including adding some 
of the penalty terms in the eigenvalue problem, produced less satisfactory results.

As we observe from our numerical results that 
the penalty does not need to be very large. Even though the boundary error decreases, it is
much lower than the interior error for relatively small penalty terms.
Thus, if a small penalty term is preferred (e.g., for time-dependent problems),
then one can attempt to find an optimal penalty. For the steady
state problem considered in this paper, we do not investigate
this problem. 
We believe  that other appropriate techniques for
coupling discontinuous spectral basis functions should be also 
considered. For instance, formulations that avoid the use of
artificial penalty terms and use approaches such as hybridized discontinuous
Galerkin methods, \cite{DG_hybrid}, or mortar methods, \cite{Arbogast_PWY_07}. 
This is object of current research.

One can choose other multiscale spaces
such that to achieve higher accuracy with lower degrees of freedom.
For example, as we show that using proper eigenvalue problem one can
improve the error. Error to a reference solution can also be improved by selecting 
appropriate snapshot spaces. In general,  the choice of snapshot space
 depends on the input space as discussed in Introduction (see also \cite{egh12} and the local
solution space which this input space gives).
Here, we do not explore the choice of snapshot spaces since our goal is to show that 
Discontinuous Galerkin provides a nice framework to couple discontinuous
basis functions computed locally.

\section{Acknowledgements}

Y. Efendiev's work is
partially supported by the
DOE and NSF (DMS 0934837 and DMS 0811180).
J.Galvis would like to acknowledge partial support from DOE.
R. Lazarov's research was supported in parts by NSF (DMS-1016525).

This publication is based in part on work supported by Award 
No. KUS-C1-016-04, made by King Abdullah University of Science 
and Technology (KAUST).

We are grateful to Mr.\,Chak Shing Lee for implementing one of the methods and 
providing the results reported in Table
\ref{table:example2_10^4_1}.




\bibliographystyle{plain}   
\bibliography{bib_raytcho,doe_2012}
\end{document}